\documentclass[12pt,a4paper]{article}

\usepackage[cm]{fullpage}
\usepackage{amssymb}
\usepackage{amsmath}
\usepackage{amsthm}
\usepackage[english]{babel}
\usepackage[cp1250]{inputenc}
\usepackage[T1]{fontenc}
\usepackage[pdftex]{graphicx}
\usepackage{cite}

\newtheorem{thm}{Theorem}
\newtheorem{prop}{Proposition}

\title{Analytical studies of a time-fractional porous medium equation. Derivation, approximation and applications}
\author{\L ukasz P\l ociniczak\thanks{Institute of Mathematics and Computer Science, Wroclaw University of Technology, ul. Janiszewskiego 14a, 50-372 Wroclaw, Poland, e-mail: \underline{lukasz.plociniczak@pwr.edu.pl}}}
\date{}

\begin{document}
\maketitle

\begin{abstract}
In this paper we investigate the porous medium equation with a fractional temporal derivative. We justify that the resulting equation emerges when we consider the waiting-time (or trapping) phenomenon that can happen in the medium. Our deterministic derivation is dual to the stochastic CTRW framework and can include nonlinear effects. With the use of the previously developed method we approximate the investigated equation along with a constant flux boundary conditions and obtain a very accurate solution. Moreover, we generalize the approximation method and provide explicit formulas which can be readily used in applications. The subdiffusive anomalies in some porous media such as construction materials have been recently verified by experiment. Our simple approximate solution of the time-fractional porous medium equation fits accurately a sample data which comes from one of these experiments. 
\end{abstract}

\section{Introduction}
Throughout the last decade the number of experimental reports of the anomalous diffusive phenomena has grown. For example, some of these concern moisture dispersion in building materials (see for ex. \cite{Gha04,Gha12,Ram08,Tay99,Kun01}) but also some other minerals like zeolite \cite{Aze06,Aze06a}. When we consider an essentially one dimensional (other dimensions can be neglected) porous medium with one boundary kept at constant moisture, the concentration in subsequent time instant $t$ and space point $x$ is a function of $x/\sqrt{t}$. This characteristic space-time scaling is not, however, present in every type of porous media. As experiments showed, in certain materials the moisture propagates according to the $x/t^{\alpha/2}$ scaling, where $0<\alpha<2$. When $0<\alpha<1$ the moisture concentration represents a subdiffusive character, meaning that the fluid particles can be trapped in some regions for a significant amount of time. In the superdiffusive case $1<\alpha<2$ there is a possibility that water can be transported for large distances in a relatively short time (jump behavior). All of these features are certainly associated with the geometrical aspects of the medium such as pore distribution or heterogeneous features such as highly conductive channels and macropores \cite{Pac03}. Some Authors state that the anomalous behavior can also be caused by chemical reactions that casue the diffusivity to change during the imbibition \cite{Loc06}.

There have been many attempts to model the anomalous diffusion in porous media such as building materials. For example in \cite{Kun01} a model based on a nonlinear Fickian Law has been proposed. The most popular approach, however, was made by the use of the fractional derivatives \cite{Hil06,Ger10,Pac03,Sun13}. These generalizations of the classical derivative operator are necessary nonlocal. This gives an opportunity to model the history (nonlocality in time) or long-range influence in medium (nonlocality in space). The theoretical foundation that the transport in porous media should include the nonlocal phenomena was developed in \cite{Hu94,Cush93,Cush94}, where Authors devised the nonlocal Darcy's Law from the Boltzmann transport equation and argumented that the nonlocality can be a consequence of the heterogeneous pore distribution. The fractional derivative approach has an advantage to offer a theoretical framework in which all the anomalous behaviors can be governed by the derivative order. The disadvantage, however, is that the fractional equations are always more difficult to solve explicitly (either analytically or numerically) especially in the nonlinear case. 

In this work we generalize our previous results concerning the time-fractional porous medium equation, which is used to model anomalous moisture dispersion. The initial results were published in \cite{Plo13} and further developed in \cite{Plo14}. In what follows we present a theoretical derivation of the time-fractional anomalous diffusion equation and provide an analysis of its approximate solutions. The boundary conditions, that we supplement, are of the two types: constant concentration or constant flux at the interface. Both of these have a direct physical interpretation. We present our approximation method and show some estimates on the error terms. All of our theoretical considerations are illustrated by numerical analysis and fitting with real experimental data. Both of these verify that our approximations are sensibly accurate when describing anomalous diffusion in porous media. Moreover, a simple and closed form of the formulas we derive is an additional advantage which can quicken the data fitting process in applications. The usual way of data fitting is being done with a numerical solution of the nonlinear anomalous diffusion equation. Due to the nonlocality of the fractional operator, the computational power needed for the nonlinear finite difference method is very large. Simple (but approximate) solutions that we find can be a remedy for these computational problems. 

\section{Model of the time-fractional anomalous diffusion}
\subsection{Formulation}
Consider a porous medium in which the Darcy's Law holds \cite{Bear}
\begin{equation}
	q=-\frac{\rho k}{\mu} \nabla p,
\end{equation}
where $q$ is the flux of the fluid with viscosity $\mu$ flowing under the pressure $p$ through the medium of permeability $k$. In any porous medium not all of its space can be filled with the flow. The \emph{concentration} $u=u(x,t)$ is defined to be density of the water inside the Representative Elementary Volume (REV) at position $x$ and time instant $t$. The total fraction of space available for the fluid is called the \emph{porosity} of the medium. Under suitable assumptions \cite{Szym} the pressure $p$ can be expressed as a monotone function of the concentration $u$, so we can write
\begin{equation}
	q=-\frac{\rho k}{\mu} \frac{dp}{du} \; \nabla u =: -D(u) \nabla u,
\label{Darcy}
\end{equation}
where $D$ denotes the diffusivity, which can depend on the moisture concentration $u$. In applications this is almost always the case \cite{Szym} - the diffusivity can change even a several orders of magnitude with the change of $u$. In any region of the porous medium the amount of fluid must be conserved and thus the moisture must obey the continuity equation
\begin{equation}
	u_t + \nabla\cdot q=0,
	\label{conti}
\end{equation}
where subscript denotes the derivative with respect to time variable $t$. Above equation along with (\ref{Darcy}) gives us the governing nonlinear diffusion equation known in the hydrology as the \emph{Richards equation} \cite{Rich}
\begin{equation}
	u_t = \nabla \cdot \left(D(u) \nabla u\right),
	\label{Richards}
\end{equation}
If we included the gravity we would obtain an additional convective term. But in our model it is sufficient to consider only the diffusive phenomena. For our purposes it is also sufficient to reduce the equation to one spatial dimension since the diffusion in the experiment we want to model progresses mostly in one direction. In this work we will consider two different initial-boundary conditions for (\ref{Richards}), namely
\begin{equation}
	u(0,t)=C, \quad u(x,0)=0, \quad x>0, \; t>0,
	\label{cond1}
\end{equation}
which models a constant concentration at the face of an initially dry sample and 
\begin{equation}
	-D\left(u(0,t)\right)u_x(0,t)=Q, \quad u(x,0)=0, \quad x>0, \; t>0,
	\label{cond2}
\end{equation}
which is a mathematical description of constant flux at the origin.

As recent investigations show, equation (\ref{Richards}) is not sufficient to describe the moisture distribution in some porous media, in particular construction materials. Namely, the characteristic space-time scaling of some processes is very different from the solution of the classical diffusion equation. By a quick calculation we can check that a solution of (\ref{Richards}) with conditions (\ref{cond1}) is a function of $x/\sqrt{t}$ only. In this case (\ref{Richards}) reduces to an ordinary differential equation. The experiments from \cite{Gha04,Ram08} show that the solution of the equation, which is supposed to describe the (anomalous) diffusion, scales as $x/t^{\alpha/2}$ for $\alpha\in(0,2)$.

To model the anomalous behavior of the characteristic space-time scaling we use the fractional partial differential equation \cite{Plo13,Plo14,Aze06,Hil06}
\begin{equation}
	^C \partial_t^\alpha u = \left( D(u) u_x \right)_x, \quad 0<\alpha<1,
\label{E}
\end{equation}
where $^C \partial_t^\alpha$ is the Caputo partial derivative operator defined on a suitably chosen function space by (\cite{Pod99,Kil06})
\begin{equation}
	^C \partial_t^\alpha u(x,t)=\frac{1}{\Gamma(n-\alpha)}\int_0^t (t-s)^{n-\alpha-1} \frac{\partial^n u}{\partial t^n} (x,s)ds,
	\label{Caputo}
\end{equation}
where $n=[\alpha]+1$ is the first integer greater than $\alpha$. This equation describes the subdiffusive character of the porous medium. In that case the fluid particles can be trapped in some regions for a prolonged periods of time. This in consequence leads to the memory effect of the process. At some particular time, not only the fluid driven by the flow can cross the REV but also some of its portions that have just been freed from a longer wait periods. This prolonged waiting times can be consequences of many physical phenomena that can happen in the porous medium.

It is important to make the following remark. It is a well known fact that with vanishing initial conditions, that is for $\partial^k u/\partial t^k(x,0)=0, \; k=0,...,[\alpha]$, the Caputo fractional derivative (\ref{Caputo}) coincides with the Riemann-Liouville version (see \cite{Pod99} and \cite{Gor00} for the self-similar case) defined by the formula
\begin{equation}
	\partial_t^\alpha u(x,t)=\frac{1}{\Gamma(n-\alpha)}\frac{\partial^n}{\partial t^n}\int_0^t (t-s)^{n-\alpha-1} u(x,s)ds.
\end{equation}
where $n=[\alpha]+1$ is the first integer greater than $\alpha$. This operator requires less regularity on the differentiated function and thus is more suited to be treated mathematically. The difference between the Caputo and Riemann-Liouville fractional derivatives can be easily seen by defining the fractional integral \cite{Pod99}
\begin{equation}
	I_t^{\alpha}u(x,t)=\frac{1}{\Gamma(\alpha)}\int_0^t (t-s)^{\alpha-1} u(x,s)ds, \quad {\alpha>0}.
\end{equation}
Then we have the representation $\partial_t^\alpha = \partial_t I_t^{1-\alpha}$ and $^C \partial_t^\alpha =  I_t^{1-\alpha} \partial_t$, that is they differ by the order of differentiation operator. It is also worth to mention the composition formula for the functions with vanishing initial conditions
\begin{equation}
	I^{\alpha}\partial_t^\alpha u(x,t)=u(x,t) \quad \text{if} \quad \frac{\partial^k u}{\partial t^k}(x,0)=0, \; k=0,...,[\alpha].
	\label{compo}
\end{equation}
With some subtleties, the fractional derivative operators enjoy many other identities that are similar to the integer order derivatives. For a thorough and deep analysis of many topics concerning these matters see \cite{Kil06}. 

The initial values of $u$ in (\ref{cond1})-(\ref{cond2}) are zero for our case $0<\alpha<1$. Therefore in our analysis we will consider the following equation for the time-fractional nonlinear diffusion equation
\begin{equation}
	\partial_t^\alpha u = \left( D(u) u_x \right)_x, \quad x>0, \; t>0, \quad 0<\alpha<1,
\label{ERL}
\end{equation}
We choose the particular form of the diffusivity $D$ to be proportional to the some power of the concentration, that is $D(u)=D_0 u^m$. This is a well-known choice which accurately describes many different porous media. In \cite{Aze06} Authors inverted (\ref{ERL}) and derived a formula for diffusivity. After calculations they obtained a profile similar to our power-type choice of diffusivity function. To simplify the final form of the investigated equation we introduce the proper scales
\begin{equation}
	x^*=\frac{x}{L}, \quad t^*=\frac{t}{T}, \quad u^*=\frac{u}{C}, \quad T=\left(\frac{L^2}{D_0 C^m}\right)^{\frac{1}{\alpha}},
	\label{scaling1}
\end{equation}
in the case of (\ref{cond1}) and 
\begin{equation}
	x^*=\frac{x}{L}, \quad t^*=\frac{t}{T}, \quad u^*=\frac{u}{(LQ)^{1/(m+1)}}, \quad T=\left(\frac{L^2}{D_0 (LQ)^{m/(m+1)}}\right)^{\frac{1}{\alpha}},
	\label{scaling2}
\end{equation}
for (\ref{cond2}). Then, the equation describing the time-fractional nonlinear diffusion takes the form
\begin{equation}
	\partial_t^\alpha u = \left( u^m u_x \right)_x, \quad x>0, \; t>0, \quad 0<\alpha<1,
\label{ERL2}
\end{equation}
where we dropped the asterisks for the clarity of presentation. The initial-boundary conditions are now
\begin{equation}
	u(0,t)=1, \quad u(x,0)=0, \quad x>0, \; t>0,
\label{cond3}
\end{equation}
or
\begin{equation}
	-u(0,t)^m u_x(0,t)=1, \quad u(x,0)=0, \quad x>0, \; t>0,
	\label{cond4}
\end{equation}
Equation (\ref{ERL2}) can be called the time-fractional porous medium equation. For the mathematical analysis of its space-fractional version the Reader is referred to \cite{Bil13,Pab11}. 

\subsection{Derivation}
Equation (\ref{E}) modeling subdiffusion, is almost always derived by stochastic analysis in a Continuous Time Random Walk (CTRW) framework \cite{Met00}. It is instructive to see how this equation can be devised deterministically without any assumptions on the probability distributions of waiting times. 

Once again consider a fluid in the porous medium which flow is governed by the continuity equation (\ref{conti}). Assume that due to certain physical phenomena fluid particles can be trapped in some regions for a period of time $s$. Due to this trapping, only a portion of the instantaneous flux at time $t$ should be taken into account when calculating the change of concentration. The rest consists of the particles that have just been released. In that case the continuity equation can be written as a delayed equation with weights some $w_0$ and $w_1$
\begin{equation}
	u_t=-\left(w_0\nabla \cdot q(x,t)+w_1\nabla \cdot q(x,t-s)\right).
\end{equation}
At instant $t$ only a fraction of the flux comes through the REV, the rest if trapped for $s$. If such trapping can happen for a several periods $s_1,...,s_n$ we then have
\begin{equation}
	u_t=-\sum_{i=1}^n w_i \; \nabla\cdot q(x,t-s_i).
\end{equation}
By introducing the weight density $w=w(s)$ it is straightforward to generalize trapping phenomena to a continuous distribution of periods. By definition we have $w_i=w(s_i)\Delta s_i$, where $\Delta s_i=s_i-s_{i-1}$ and thus
\begin{equation}
	u_t=-\sum_{i=1}^n w(s_i) \; \nabla\cdot q(x,t-s_i) \Delta s_i \longrightarrow -\int_0^t w(s) \; \nabla\cdot q(x,t-s) ds.
\end{equation}
By a change of the variable we arrive at a non-local integro-differential equation in which any past value of the flux is taken into account
\begin{equation}
	u_t=-\int_0^t w(t-s) \; \nabla\cdot q(x,s) ds.
	\label{nonlocal}
\end{equation}
This accounts for the fluid particles that can be trapped for any period of time. The amount of flux of the particles that wait for the time $s$ is equal to $w(s)$. In the sequel we will determine the form of the (generalized) function $w$. To this end we make the following claims. 
\begin{enumerate}
	\item (\textit{Generalization}) In the case of the classical diffusion, that is without the waiting times, (\ref{nonlocal}) should reduce to (\ref{conti}).
	\item (\textit{Simplicity}) The form of $w$ should be as simple as possible.
\end{enumerate}
Because we assume that the resulting equation should possess a self-similar solutions dependent on $\alpha$, let us denote this explicitly by writing $w_\alpha$ instead of $w$. By the first claim we should have $w_\alpha\rightarrow \delta$, when (say) $\alpha\rightarrow 1$. Here $\delta$ is the Dirac delta distribution. The limit should be understood in the appropriate weak sense in the distribution topology. It is easy to see that when $w=\delta$ we immediately retrieve the classical local equation (\ref{conti}). As $\alpha\rightarrow 1$ the particles are being trapped for shorter and shorter instants eliminating the non-local memory effect. Since $w_\alpha$ should be treated as a distribution we will make a slight change of the kernel in (\ref{nonlocal}) in order to treat only ordinary functions. The simplest approximation of the $\delta$ distribution is the following 
\begin{equation}
	w_1=\delta=\lim_{h\rightarrow 0} \frac{\chi_{[0,h]},}{h},
\end{equation}
where the limit is understood in the weak sense, that is
\begin{equation}
\begin{split}
	 \nabla\cdot q(x,t) =& \int_0^t w_1(t-s) \nabla\cdot q(x,s) ds = \int_0^t \delta(t-s) \nabla\cdot q(x,s) ds \\ 
	=&\lim_{h\rightarrow 0} \int_{t-h}^t \nabla\cdot q(x,s) ds=\frac{\partial}{\partial t} \int_0^t \nabla\cdot q(x,s) ds.
\end{split}
\end{equation}
The above equation is just the elementary formula for differentiating an integral as a function of the upper limit of integration. This motivates us to rewrite the equation (\ref{nonlocal}) as 
\begin{equation}
	u_t=-\frac{\partial}{\partial t}\int_0^t k_\alpha(t-s) \; \nabla\cdot q(x,s) ds,
	\label{nonlocal2}
\end{equation}
where $k_\alpha$ is now a \emph{function} with $k_1\equiv 1$. The limit of $\delta$ distribution is now hidden under the derivative operator. It is easy to make a sensible educated guess on the form of $k_\alpha$: by the simplicity claim we suppose that it should be a power function $k_\alpha(s)=C_\alpha s^{\alpha-1}$, where the constant $C_\alpha$ has to be determined. We see that if we choose $C_\alpha=1/\Gamma(\alpha)$ we will get
\begin{equation}
	u_t=-\frac{1}{\Gamma(\alpha)}\frac{\partial}{\partial t}\int_0^t (t-s)^{\alpha-1} \; \nabla\cdot q(x,s) ds=-\partial_t^{\alpha-1} \nabla\cdot q,
	\label{nonlocal3}
\end{equation}
where we identified the Riemann-Liouville fractional derivative of order $1-\alpha$. Different choice of $C_\alpha$ would yield a constant multiple of the fractional derivative and would not modify the quantitative features of the equation. Thus our choice of $C_\alpha$ is the simplest. Now, by the composition identity (\ref{compo}) we can operate with $I_t^{1-\alpha}$ on both sides of (\ref{nonlocal3}) and get
\begin{equation}
	^C \partial_t^{\alpha} u=-\nabla\cdot q,
	\label{nonlocal4}
\end{equation}
which is the same as (\ref{E}) with flux from the Darcy's Law (\ref{Darcy}). This concludes the derivation. We see that the emergence of the fractional derivative is very natural provided that the fluid particles can be trapped in some regions in space for a period of time with a power-type weight. 

\section{Analytical results}
\subsection{Self-similar setting}
In this paper we are mainly concerned with the self-similar solutions of (\ref{ERL2})
\begin{equation}
	u(x,t)=t^a U(\eta), \quad \eta=\frac{x}{t^b},
	\label{scale}
\end{equation}
where $a$ and $b$ have to be determined. The fractional derivative transforms as
\begin{equation}
\begin{split}
	\partial^\alpha_t u(x,t)&=\frac{\partial}{\partial t} \left( I^{1-\alpha}_t \left(t^a U(x t^{-b})\right) \right)=\frac{1}{\Gamma(1-\alpha)}\frac{\partial}{\partial t} \int_0^t (t-z)^{-\alpha} z^a U(x z^{-b})dz\\
	&=\frac{1}{\Gamma(1-\alpha)}\frac{\partial}{\partial t} \left(t^{a-\alpha+1} \int_0^1 (1-s)^{-\alpha} s^a U(\eta s^{-b})ds\right).
\end{split}
\label{trans1}
\end{equation}
If we introduce the Erd\'{e}lyi-Kober fractional integral operator (see \cite{Kir93})
\begin{equation}
	I^{\beta,\gamma}_\delta U(\eta) = \frac{1}{\Gamma(\gamma)}\int_0^1 (1-z)^{\gamma-1} z^\beta U(\eta z^\frac{1}{\delta}) dz,
	\label{EK}
\end{equation}
then equation (\ref{trans1}) can be restated as
\begin{equation}
	\partial^\alpha_t u(x,t)=\frac{\partial}{\partial t} \left(t^{a-\alpha+1} I^{a, 1-\alpha}_{-\frac{1}{b}} U(\eta)\right) = t^{a-\alpha}\left[(a-\alpha+1)-b\eta \frac{d}{d\eta}\right]I^{a, 1-\alpha}_{-\frac{1}{b}} U(\eta),
\label{trans15}
\end{equation}
where we used the chain-rule for (\ref{scale})
\begin{equation}
	\frac{\partial}{\partial t} = \frac{\partial\eta}{\partial t}  \frac{d}{d\eta}=-b\eta t^{-1}  \frac{d}{d\eta}.
\end{equation}
In our self-similar setting, the spatial derivative term in (\ref{ERL2}) changes into
\begin{equation}
	\left( u^m(x,t) u_x(x,t) \right)_x = t^{a(m+1)-2b} \frac{d}{d\eta}\left(U^m(\eta) \frac{d}{d\eta}U(\eta)\right).
\label{trans4}
\end{equation}
Now, by equating (\ref{trans15}) with (\ref{trans4}) we eliminate the explicit time dependence if and only if
\begin{equation}
	2b-m a = \alpha.
\label{war1}
\end{equation}
The second equation for the unknown constants $a$ and $b$ is to be determined from the conditions (\ref{cond3}) or (\ref{cond4}). In the case of (\ref{cond3}) we have $1=u(0,t)=t^a U(0)$ and thus $a=0$. From (\ref{war1}) we immediately get $b=\alpha/2$. The other condition now has the form $0=u(x,0)=U(\infty)$. In the classical versions of the problems we are considering its can be shown that the solution always possess a \emph{compact support}. That is, there exists $\eta^*$ such that $U(\eta)=0$ for every $\eta\geq \eta^*$ \cite{Gil76}. This feature is also very appealing from the physical point of view indicating the finite speed of propagation. The derivative at $\eta^*$ should be treated in a weak sense. In what follows we will seek for solutions with compact support only. Considering the (\ref{cond4}) we have $-1=u(0,t)^m u_x(0,t)=t^{a(m+1)-b}U(0)^m U'(0)$ which forces $b=a(m+1)$ what along with (\ref{war1}) yields $a=\alpha/(m+2)$ and $b=\alpha (m+1)/(m+2)$. Moreover, the condition $u(0,x)=0$ is satisfied automatically. 

To sum up, our self-similar problem consists of the following equations
\begin{equation}
	\left(U^m U'\right)' = \left[(a-\alpha+1)-b\eta \frac{d}{d\eta}\right]I^{a, 1-\alpha}_{-\frac{1}{b}} U, \quad 0<\alpha< 1,
\label{mainEq}
\end{equation} 
along with two sets of boundary conditions 
\begin{equation}
	U(0)=1, \quad U(\eta)=0 \quad \text{for} \quad \eta\geq\eta^*>0 \quad \text{and} \quad a=0, \quad b=\frac{\alpha}{2},
\label{conds1}
\end{equation}
or
\begin{equation}
	-U(0)^m U'(0)=1, \quad U(\eta)=0 \quad \text{for} \quad \eta\geq\eta^*>0 \quad \text{and} \quad a=\frac{\alpha}{m+2}, \quad b=\frac{m+1}{m+2}\;\alpha,
\label{conds2}
\end{equation}
for some $\eta^*$ which defines the compact support of and has to be determined as a part of the solution. Notice also that for every $\alpha\in(0,1]$ the formula (\ref{mainEq}) represents a \emph{second order} integro-differential equation. This along with the two conditions in each set (\ref{conds1})-(\ref{conds2}) provides a well-posed nonlinear problem. The solution of these equations can be obtained in the closed form only in the linear case, i.e. if $m=0$ \cite{Gor00}. For $m>0$ even in the classical setting $\alpha=1$, only approximate solutions are known \cite{Lis80,Kin88}. In what follows we will convert (\ref{mainEq}) to an approximate differential equation and then solve it. 

\subsection{Approximation}
The structure of the equation (\ref{mainEq}) is very complex. It combines both local and nonlocal behaviors of the function $U$ which prevents it from being easily solved and analyzed analytically. In order to obtain some information about the behavior of $U$ we embrace a technique developed in \cite{Plo13} and later improved in \cite{Plo14}. It has its foundation on asymptotic techniques such as perturbation theory and Laplace method for asymptotic integrals. The strategy is to approximate the Erd\'{e}lyi-Kober fractional integral (\ref{EK}) by a derivative operator and then to tackle the resulting ordinary differential equation, which resembles the one considered in the classical case. Since the classical theory for a porous medium equation is now mostly developed we retranslate our problem into the similar one which lies on the firm ground. 

The main approximation technique that we will utilize is based on the following theorem proved in \cite{Plo14}. These results state that for sufficiently decent functions the Erd\'{e}lyi-Kober operator can be restated as a series of derivatives. A more thorough and applicable research was conducted in \cite{Poo12,Luk14}, where Authors elaborated on the accuracy of approximation and presented many applications.

\begin{thm}
\label{Tw1}
	Let $U$ be an analytic function and $a>-1$, $b>0$, $c>0$. Then the following representation holds
	\begin{equation}
		I^{\beta,\gamma}_\delta U(\eta)=\sum_{k=0}^\infty \lambda_k U^{(k)}(\eta) \frac{\eta^k}{k!},
		\label{appr}
	\end{equation}
	where
	\begin{equation}
		\lambda_k=\frac{(-1)^k}{\Gamma(\gamma)} \int_0^1 (1-s)^{\gamma-1} s^\beta (1-s^{1/\delta})^k ds=\sum_{j=0}^{k}\binom{k}{j}(-1)^{k-j} \frac{\Gamma\left(\beta+\frac{j}{\delta}+1\right)}{\Gamma\left(\beta+\gamma+\frac{j}{\delta}+1\right)}.
		\label{lamb}
	\end{equation}
	Moreover, we have the asymptotic relation when $k\rightarrow\infty$
	\begin{equation}
		\lambda_k \sim(-1)^k \frac{\Gamma\left(\delta(\beta+1)\right)}{\Gamma(\gamma)} \frac{\delta}{k^{\delta(\beta+1)}}.
		\label{lambas}
	\end{equation}
\end{thm}

In order to approximate the E-K operator we can terminate the series (\ref{appr}) at some fixed $N$. This is also useful if the function $U$ is not analytic. The approximation should still be valid. The proof of the Theorem \ref{appr} uses the argument similar to the proof of Watson's lemma associated with Laplace method for asymptotic integrals. Since the integrand (\ref{EK}) is singular near $z=1$ we can expect that the mass of whole integral is located near this point. Hence, it is reasonable to substitute $U(\eta)$ instead of $U(\eta z^{1/\delta})$. Then, we would obtain the zero-order term in (\ref{appr}). The whole decomposition comes from the expansion into Taylor series near $z=1$.  

We can obtain some qualitative estimates on the rate of convergence of (\ref{appr}). First, notice that $(-1)^k \lambda_k$ are strictly decreasing. It follows from the positivity of the integral in (\ref{lamb}) and the fact that $(1-z^{1/c})^k$ is decreasing with $k$. By the formula (\ref{lambas}) we see that $\lambda_k$ converge to zero with a rate of a power type (larger $\beta$ and $\delta$ give faster convergence). Moreover, by some elementary calculation we can deduce the following result.

\begin{prop}
The following statement holds
\begin{equation}
	\lambda_k \stackrel{\gamma\rightarrow 0^+}\longrightarrow \left\{
	\begin{array}{lc}
		1, & k=0;\\
		0, & k>0,
	\end{array}
	\right.
	\quad \text{for} \quad \beta>-1, \; \delta>0.
	\label{lambb}
\end{equation}
\end{prop}

\begin{proof}
To see this, first consider the case $k=0$. Using the definition of the Beta function and its representation in terms of the Gamma functions we obtain 
\begin{equation}
	\lambda_0 = \frac{1}{\Gamma(\gamma)}\text{B}(\beta+1,\gamma)= \frac{\Gamma(\beta+1)\Gamma(\gamma)}{\Gamma(\gamma)\Gamma(\beta+\gamma+1)}\stackrel{\gamma\rightarrow 0^+}\longrightarrow 1.
\end{equation}
Now, to investigate the case of $k>0$ we must consider different domains of $\beta$ separately. For $\beta>0$ by integrating by parts we obtain
\begin{equation}
\begin{split}
	\lambda_k&=\frac{(-1)^k}{\Gamma(\gamma)}\left(\left[-\frac{1}{\gamma}(1-z)^\gamma z^\beta (1-z^{1/\delta})^k\right] _0^1+\frac{1}{\gamma}\int_0^1 (1-z)^\gamma \frac{d}{dz}\left(z^\beta(1-z^{1/\delta})^k\right)dz\right) \\
	&=\frac{(-1)^k}{\Gamma(\gamma+1)}\int_0^1 (1-z)^\gamma \frac{d}{dz}\left(z^\beta(1-z^{1/\delta})^k\right)dz \stackrel{\gamma\rightarrow 0^+}\longrightarrow (-1)^k \left[z^\beta (1-z^{1/\delta})^k\right]_0^1=0,
\end{split}
\end{equation}
where we used the Dominated Convergence Theorem. If $\beta=0$ then by a similar calculation we have
\begin{equation}
	\lambda_k=\frac{(-1)^k}{\Gamma(\gamma+1)}\left(1+\int_0^1 (1-z)^\gamma \frac{d}{dz}\left((1-z^{1/\delta})^k\right)dz\right) \stackrel{\gamma\rightarrow 0^+}\longrightarrow(-1)^k\left(1+\left[(1-z^{1/\delta})^k\right]_0^1\right),
\end{equation}
which vanishes for $k>0$. Finally, when $-1<\beta<0$ the integration by parts yields the following result
\begin{equation}
	\lambda_k=-\frac{(-1)^k}{\Gamma(\gamma)}\int_0^1 \left(\int_0^z (1-t)^{\gamma-1} t^\beta dt\right) \frac{d}{dz}\left((1-z^{1/\delta})^k\right)dz.
\end{equation}
The inner integral is an incomplete Beta function B$(z;\beta+1,\gamma)$. To calculate the limit notice that
\begin{equation}
\begin{split}
	\frac{1}{\Gamma(\gamma)} & \text{B}(z;\beta+1,\gamma)=\frac{1}{\Gamma(\gamma)}\left(\text{B}(\beta+1,\gamma)-\int_z^1 (1-t)^{\gamma-1}t^\beta dt\right) = \frac{\Gamma(\beta+1)}{\Gamma(\beta+\gamma+1)}\\ 
	&+\frac{1}{\Gamma(\gamma+1)}(1-z)^\gamma z^\beta-\frac{1}{\Gamma(\gamma+1)}\int_z^1 (1-t)^\gamma \frac{d}{dz}\left(t^\beta\right)dt \stackrel{\gamma\rightarrow 0^+}\longrightarrow 1-z^\beta-\left[t^\beta\right]_z^1=0.
\end{split}
\end{equation}
\end{proof}
This result show that when $\gamma$ becomes closer to 0 only the $\lambda_0$ survives. Recalling our main equation (\ref{mainEq}) we see than in its case we have $\gamma=1-\alpha$. Thus, with $\alpha$ going to $1$ we reobtain the classical diffusion equation. This shows that if we approximated the E-K operator in (\ref{mainEq}) by the $\lambda_0$ term, we would be utilizing a similar technique to the first-order perturbation theory with respect to $\alpha$. The closer $\alpha$ to the integer values the better the approximation.   

Now we turn to the analysis of the approximation error. The difference between the E-K operator and the approximating series (\ref{appr}) is
\begin{equation}
	I^{\beta,\gamma}_\delta U(\eta) - \sum_{k=0}^{N-1} \lambda_k U^{(k)}(\eta)\frac{\eta^k}{k!}= \lambda_N U^{(N)}(\eta)\frac{\eta^N}{N!} + \cdots,
\label{errest1}
\end{equation}
which means that the convergence depends on the nature of $U$ (and its derivatives). The series should represent the E-K operator very well especially for small $\eta$ and $U^{(k)}$ bounded for $k\in\mathbb{N}$. By using the asymptotic form of $\lambda_k$ as in (\ref{lambas}) we have
\begin{equation}
	I^{\beta,\gamma}_\delta U(\eta) - \sum_{k=0}^{N-1} \lambda_k U^{(k)}(\eta)\frac{\eta^k}{k!}\sim(-1)^N \delta\frac{\Gamma\left(\delta(\beta+1)\right)}{\Gamma(\gamma)} \frac{U^{(N)}(\eta)}{N^{\delta(\beta+1)}} \frac{\eta^N}{N!} \quad \text{as} \quad N\rightarrow\infty.
\end{equation} 
Wee see that for fixed $\eta$ this convergence is very fast: it goes as $(N^{\delta(\beta+1)}N!)^{-1}$ for $U^{(k)}$ uniformly bounded for $k\in\mathbb{N}$. For example take $U(\eta)=\exp(-\eta)$, then
\begin{equation}
	I^{\beta,\gamma}_\delta U(\eta) - e^{-\eta}\sum_{k=0}^{N-1} (-1)^k \lambda_k \frac{\eta^k}{k!}\sim \delta\frac{\Gamma\left(\delta(\beta+1)\right)}{\Gamma(\gamma)} \frac{e^{-\eta}}{N^{\delta(\beta+1)}} \frac{\eta^N}{N!} \quad \text{as} \quad N\rightarrow\infty.
\label{errest2}
\end{equation}
Notice also that for any fixed $N$ this error vanishes for $\eta\rightarrow 0$ or $\eta\rightarrow\infty$. As a side result, by the definition of the E-K operator (\ref{EK}) we obtain a representation for the generating function for the $\lambda_k$ series
\begin{equation}
	\sum_{k=0}^\infty \lambda_k \frac{\eta^k}{k!}=\frac{e^{-\eta}}{\Gamma(\gamma)}\int_0^1 (1-z)^{\gamma-1} z^\beta e^{\eta z^{1/\delta}} dz.
\end{equation}
The numerical illustration of these approximations is presented on Fig. \ref{RelErr}. For the function $e^{-\eta}$ we plot the relative error of approximating the $I^{\beta,\gamma}_\delta$ operator with the series (\ref{appr}), that is
\begin{equation}
	\text{Relative error}(N):=\frac{I^{\beta,\gamma}_\delta U(\eta) - \sum_{k=0}^{N-1} \lambda_k U^{(k)}(\eta)\frac{\eta^k}{k!}}{I^{\beta,\gamma}_\delta U(\eta)}.
	\label{RelErrDef}
\end{equation}
We can see that the error estimates (\ref{errest1}) and (\ref{errest2}) and the approximation are decently accurate even for a small number of approximating terms. The error grows with $\eta\rightarrow\infty$ but on compact sets stays bounded. 

\begin{figure}[htb!]
	\centering
	\includegraphics[scale=0.7]{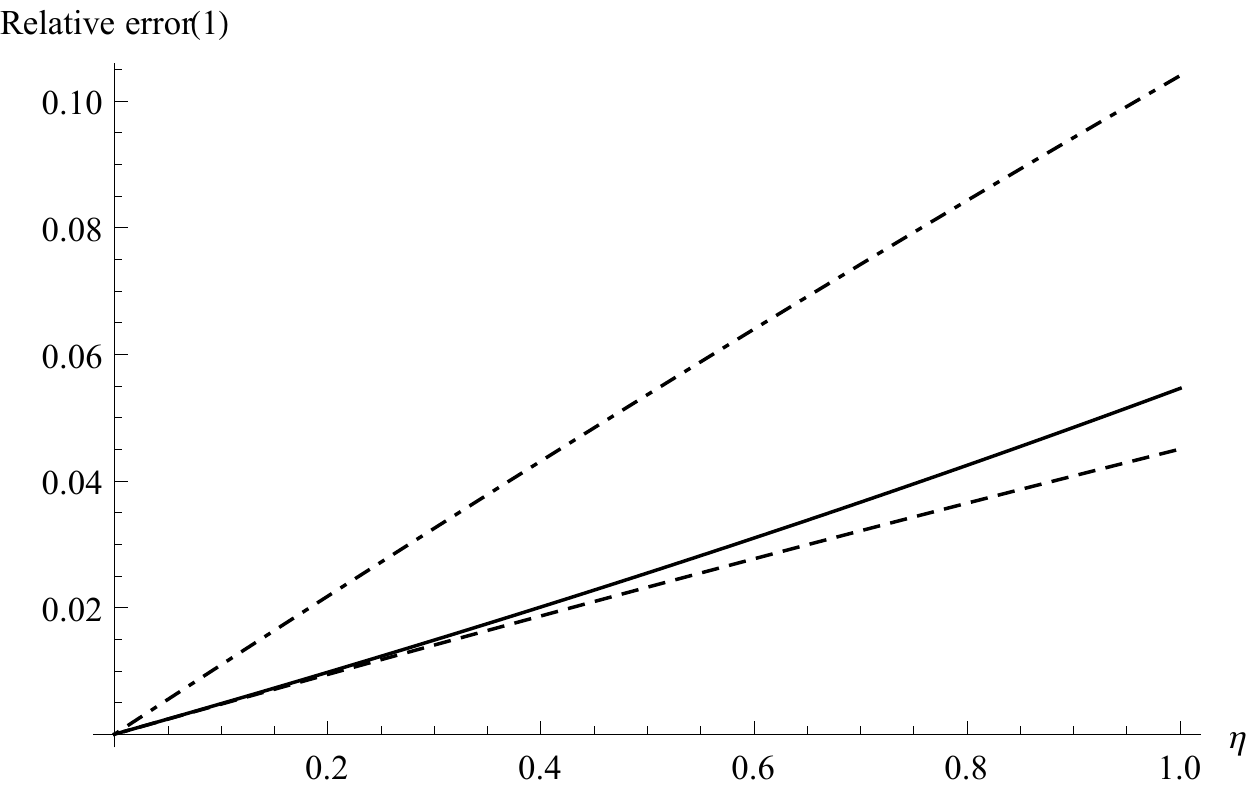}
	\includegraphics[scale=0.7]{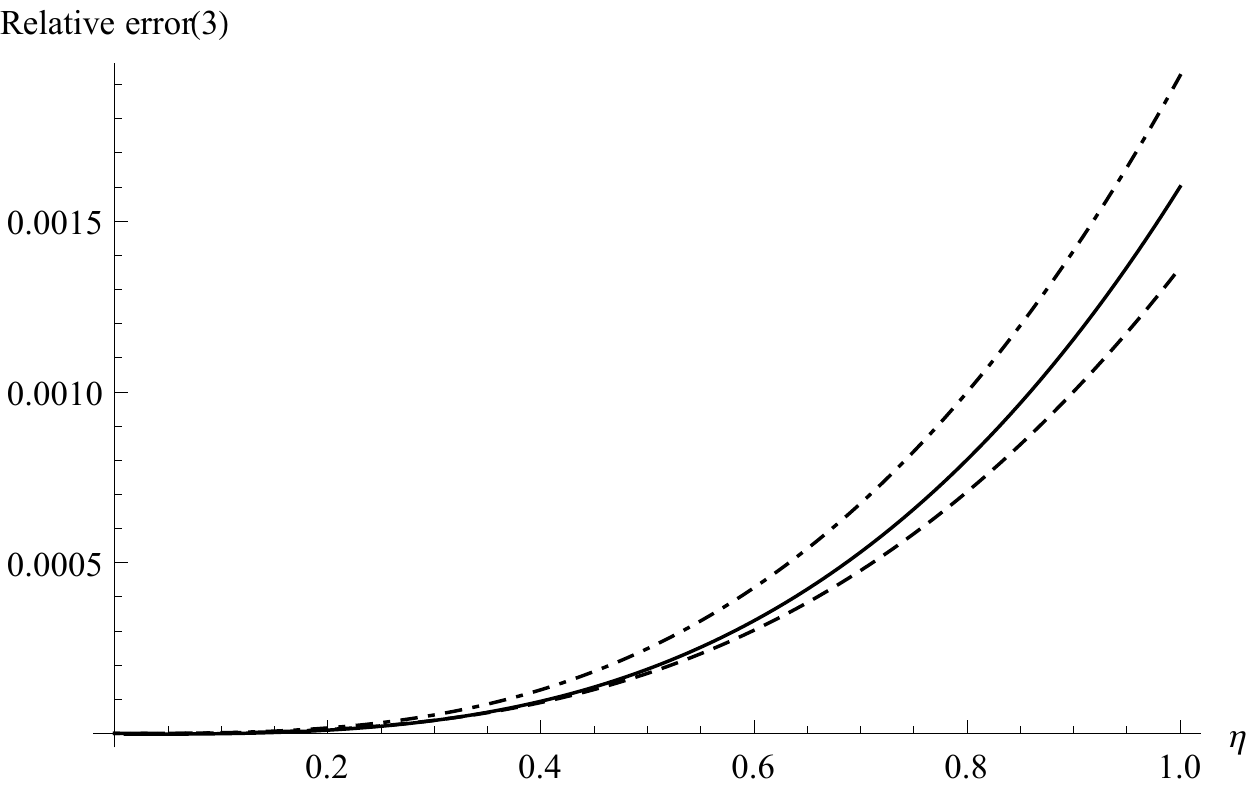}
	\caption{Relative errors of the approximation (\ref{errest1}) of $U(\eta)=e^{-\eta}$ for $N=1$ (on the left) and $N=3$ (on the right) calculated for $\beta=1$, $\gamma=0.1$ and $\delta=1$. Solid line: the relative error (\ref{RelErrDef}); dashed line: the first term on right-hand side of (\ref{errest1}); dot-dashed line: asymptotic approximation (\ref{errest2}).}
	\label{RelErr}
\end{figure}

The rigorous results concerning functions with uniformly bounded derivatives is contained in the proposition below.
\begin{prop}
	Let $U$ be a C$^N$ function with uniformly bounded derivatives. Then we have the following estimate
	\begin{equation}
	\left| I^{\beta,\gamma}_\delta U(\eta) - \sum_{k=0}^{N-1} \lambda_k U^{(k)}(\eta) \frac{\eta^k}{k!}\right|\leq C \frac{|\eta|^N}{N^{\delta(\beta+1)}N!},
	\end{equation}
	where $\lambda_k$ is as in (\ref{lamb}) and $C$ depends on $U$, $\beta$, $\gamma$ and $\delta$.
\end{prop}
\begin{proof}
We expand the function $f_\eta(y):=U(\eta y)$ into the series at $y=1$ with $N$th remainder
\begin{equation}
	f_\eta(y)=\sum_{k=0}^{N-1} U^{(k)}(\eta) \frac{\eta^k}{k!} (y-1)^k+\frac{U^{(N)}(\eta \zeta)}{N!}\eta^N (y-1)^N,
\end{equation}
where $\zeta$ is between $1$ and $y$. By the substitution $z=y^\delta$ we have
\begin{equation}
	U\left(\eta z^{1/\delta}\right)=\sum_{k=0}^{N-1} (-1)^k U^{(k)}(\eta) \frac{\eta^k}{k!} \left(1-z^{1/\delta}\right)^k+(-1)^N\frac{U^{(N)}(\eta \zeta)}{N!}\eta^N \left(1-z^{1/\delta}\right)^N.
\end{equation}
Now, by the use of definition of E-K fractional operator (\ref{EK}) we obtain
\begin{equation}
	\left| I^{\beta,\gamma}_\delta U(\eta) - \sum_{k=0}^{N-1} \lambda_k U^{(k)}(\eta) \frac{|\eta|^k}{k!}\right|\leq \frac{\eta^N}{N!}\frac{1}{\Gamma(\gamma)}\int_0^1(1-z)^{\gamma-1}z^\beta\left(1-z^{1/\delta}\right)^N \left|U^{(N)}(\eta \zeta)\right|dz.
\end{equation}
Since, by the assumption, the derivatives of $U$ are uniformly bounded we have 
\begin{equation}
	\left| I^{\beta,\gamma}_\delta U(\eta) - \sum_{k=0}^{N-1} \lambda_k U^{(k)}(\eta) \frac{\eta^k}{k!}\right|\leq \left\|U^{(N)}\right\|_\infty |\lambda_N| \frac{|\eta|^N}{N!}.
\end{equation}
Now, due to the asymptotic relation for $\lambda_k$ (\ref{lambas}) the sequence $|\lambda_k|\; k^{\delta(\beta+1)}$ is bounded, thus
\begin{equation}
	\left| I^{\beta,\gamma}_\delta U(\eta) - \sum_{k=0}^{N-1} \lambda_k U^{(k)}(\eta) \frac{\eta^k}{k!}\right|\leq \left\|U^{(N)}\right\|_\infty |\lambda_N|\; N^{\delta(\beta+1)}  \frac{|\eta|^N}{N^{\delta(\beta+1)}N!}\leq C \frac{|\eta|^N}{N^{\delta(\beta+1)}N!}.
\end{equation}
This concludes the proof.
\end{proof}
Notice that in the above proposition we did not assume that $U$ is analytic. This means that we can investigate the convergence of the series in a compact subset of the real line for functions which have isolated points at which they do not possess a derivative. In what follows we will show that the solution of (\ref{mainEq}) with (\ref{conds1})-(\ref{conds2}) can be approximated by a function $U(\eta)=(1-\eta/\eta^*)^{1/m}$, where $\eta^*$ depends on $\alpha$. This function is not analytic and even not differentiable at $\eta=\eta^*$. Specifically, $U^{(k)}(\eta)=\frac{1}{m}\left(\frac{1}{m}-1\right)...\left(\frac{1}{m}-k+1\right) (\eta^*)^{-k} (1-\eta/\eta^*)^{1/m-k}$. Moreover
\begin{equation}
	\frac{1}{m}\left(\frac{1}{m}-1\right)...\left(\frac{1}{m}-k+1\right) = (-1)^k \frac{\Gamma(k-\frac{1}{m})}{\Gamma(1-\frac{1}{m})}\sim (-1)^k \frac{(k-1)!}{k^\frac{1}{m}\Gamma(1-\frac{1}{m})} \quad \text{as} \quad k\rightarrow\infty.
\end{equation}
And thus we have an asymptotic relation
\begin{equation}
	I^{\beta,\gamma}_\delta U(\eta) - \sum_{k=0}^{N-1} \lambda_k U^{(k)}(\eta) \frac{\eta^k}{k!}\sim \delta\frac{\Gamma\left(\delta(\beta+1)\right)}{\Gamma(\gamma)} \frac{1}{N^{1+\delta(\beta+1)+1/m}}\left(\frac{\eta}{\eta^*}\right)^N \left(1-\frac{\eta}{\eta^*}\right)^{1/m-N} \quad \text{as} \quad N\rightarrow\infty.
	\label{asymPierw}
\end{equation}
For $0\leq\eta\leq M<\eta^*$ the right-hand side of (\ref{asymPierw}) is uniformly bounded with respect to $\eta$. We see that the convergence is good, especially for $\eta$ close to $0$. 

An important remark is in order. Notice that all of our preceding considerations assume $\delta>0$ while in the equation (\ref{mainEq}) $\delta=-1/b<0$. Despite the fact that in the case of negative $\delta$ we could have not performed the integrations, the resulting $\lambda_k$ would still have sense by the analytic continuation of the Gamma function. This means that all the asymptotic formulas, such as (\ref{lambas}), would still \emph{formally} make sense for $\delta<0$. But now the convergence would be slower as $\delta$ became more negative. Taking more terms in the series (\ref{appr}) would not necessarily yield a greater accuracy. We note also that the integral (\ref{EK}) defining the E-K operator has a meaning despite the negativity of $\delta$. Because $U$ has a compact support, for sufficiently small $z$ the expression $\eta z^{1/\delta}$ is larger than $\eta^*$ and thus $U(\eta z^{1/\delta})$ vanish. Therefore, the lower integration limit is separated from $0$ (it is actually equal to $(\eta/\eta^*)^{-\delta}$) In order to approximate the E-K operator, the sensible move would be to take the $\lambda_0$ term only since it not depends on $\delta$. Moreover, as was noted before, this is also motivated by the integrand's mass concentration near a singular point. It is also easy to provide a quick estimate for the one-term approximation, which does not require any differentiability.
\begin{equation}
\begin{split}
	\left| I^{\beta,\gamma}_\delta U(\eta) - \frac{\Gamma(\beta+1)}{\Gamma(\beta+\gamma+1)} U(\eta)\right| & \leq \frac{1}{\Gamma(b)}\int_0^1 (1-z)^{\gamma-1} z^\beta \left|U(\eta z^{1/\delta})-U(\eta)\right|dz \\ 
	& \leq \frac{\Gamma(\beta+1)}{\Gamma(\beta+\gamma+1)} \sup_{z\in[0,1]} \left|U(\eta z^{1/\delta})-U(\eta)\right|,
\end{split}
\end{equation}
where we have written down the expression for $\lambda_0$ explicitly. We can see that this error strongly depends on the nature of $U$. The numerical illustration of the approximation with the $\lambda_0$ and $\lambda_1$ terms only (that is $N=1$ and $N=2$ in (\ref{errest1})) is presented on Fig. \ref{AbsErr}. The function investigated here is $U(\eta)=\sqrt{1-\eta}$ which, as will turn out in the next section, is a particular case of the approximate solution of the equation (\ref{mainEq}). We can see that because of the fact that $U'$ has a singularity at $\eta=1$, only the $\lambda_0$ term provides an accurate global approximation. Moreover, as can be seen, higher order approximations are reasonable only for small values of the argument - they diverge for $\eta$ approaching the singularity of the derivative. The choice of parameters $\beta$, $\gamma$ and $\delta$ in $I^{\beta,\gamma}_\delta$ is chosen to be the same as in (\ref{mainEq}) with (\ref{conds1}) to anticipate further results. 

\begin{figure}[htb!]
	\centering
	\includegraphics[scale=0.7]{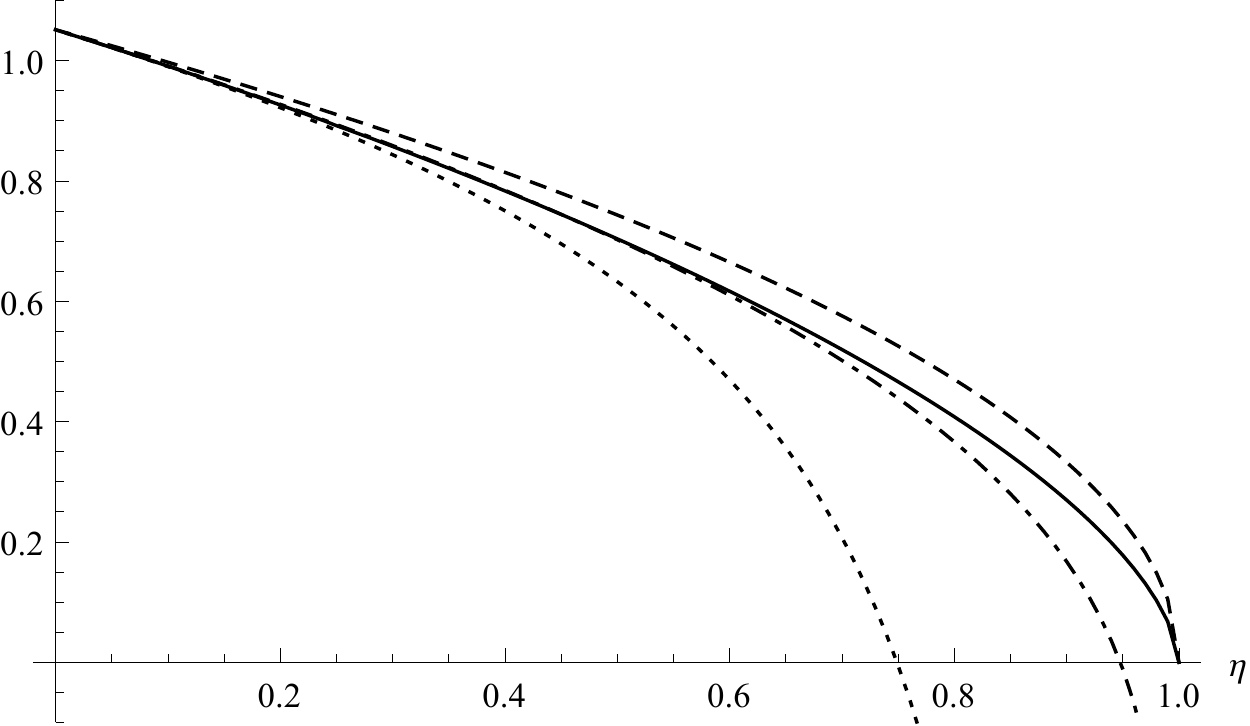}
	\includegraphics[scale=0.7]{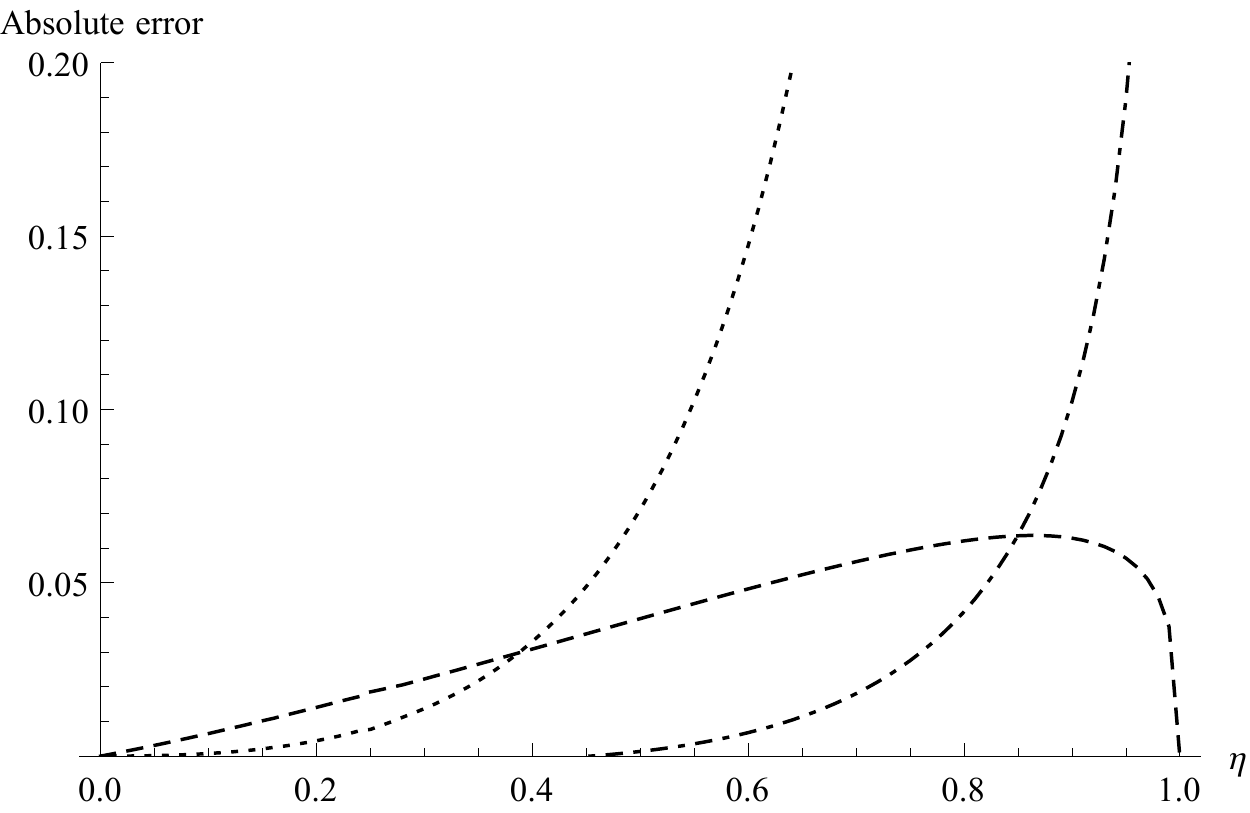}
	\caption{Comparison between the $I^{0,1-\alpha}_{-2/\alpha} U(\eta)$ for $U(\eta)=\sqrt{1-\eta}$ and its approximations (\ref{appr}) for $\alpha=0.9$. On the left: resulting function $I^{0,1-\alpha}_{-2/\alpha} U(\eta)$ (solid line) along with its approximations $\sum_{k=0}^{N-1} \lambda_k U^{(k)}(\eta)$ for $N=1$ (dashed line), $N=2$ (dot-dashed line) and $N=3$ (dotted line). On the right: absolute error $\left|I^{0,1-\alpha}_{-2/\alpha} U(\eta) - \sum_{k=0}^{N-1} \lambda_k U^{(k)}(\eta)\frac{\eta^k}{k!}\right|$ for $N=1$ (dashed line), $N=2$ (dot-dashed line) and $N=3$ (dotted line). }
	\label{AbsErr}
\end{figure}

\subsection{Solutions}
Having all the previous remarks in mind we proceed to the analysis of the main equation. We can see that for any case $0<\alpha\leq 1$ the equation (\ref{mainEq}) is of second-order. If we make the approximation $I^{\beta,\gamma}_\delta U \approx \lambda_0 U$ then any case of the equation (\ref{mainEq}) can be written in the form
\begin{equation}
	(U^m U')'\approx A U - B \eta U',
	\label{mainEqA}
\end{equation}
where the coefficients $A,B,C$ depend on $\lambda_0$, $\alpha$ and boundary conditions ((\ref{conds1}) or (\ref{conds2})). Generally $A,B$ are functions of $\alpha$ and $m$, what is summarized in the Tab. \ref{ABC1}. Equation (\ref{mainEqA}) is well known in the literature. Its mathematical theory was extensively developed in \cite{Gil76}.

\begin{table}[htb!]
\centering
\begin{tabular}{c|cc}
	&  A & B \\ \hline \\
Conditions (\ref{conds1}) & $\frac{1}{\Gamma(1-\alpha)},$ & $\frac{\alpha}{2\Gamma(2-\alpha)}$; \\ \\
Conditions (\ref{conds2}) &  $\frac{\Gamma\left(1+\frac{\alpha}{2+m}\right)}{\Gamma\left(1-\alpha+\frac{\alpha}{2+m}\right)},$ & $\alpha\;\frac{m+1}{m+2}\frac{\Gamma\left(1+\frac{\alpha}{2+m}\right)}{\Gamma\left(2-\alpha+\frac{\alpha}{2+m}\right)}$;\\ \\
\end{tabular}
\caption{Coefficients of the equation (\ref{mainEqA}).}
\label{ABC1}
\end{table}

%

The problem (\ref{mainEq}) with conditions (\ref{conds1}) was solved in \cite{Plo14}. In this work we present a more general analysis of this and the remaining problem (\ref{conds2}). We are aiming to obtain only an approximate solution of the equation (\ref{mainEq}) since it cannot be integrated (even numerically) in a straightforward way. This is due to the fact that we do not know $\eta^*$, which defines the free-boundary in (\ref{conds1})-(\ref{conds2}). This difficulty can be circumvented by a particular transformation introduced in \cite{Kin88} (but see also \cite{Lis80}) during an investigation of a similar equation describing the classical porous medium equation ($\alpha=1$). It turns out that when we substitute
\begin{equation}
	U(\eta)=\left(m (\eta^*)^2 y(z)\right)^\frac{1}{m}, \quad z=1-\frac{\eta}{\eta^*},
	\label{subs2}
\end{equation}
the free boundary problem can be transformed into an initial-value one. 

We make the final (higher-order) simplification and make the substitution (\ref{subs2}). We immediately obtain 
\begin{equation}
	U'=- (\eta^*)^{\frac{2}{m}-1} m^{\frac{1}{m}-1} y^{\frac{1}{m}-1} y', \quad (U^m U')'=(m (\eta^*)^2)^{\frac{1}{m}}\left(\frac{1}{m}y^{\frac{1}{m}-1}y'+y^{\frac{1}{m}}y''\right).
\label{sub}
\end{equation}
Hence the equation (\ref{mainEqA}) can be transformed into
\begin{equation}
	\frac{1}{m}y'^2+y y''=A y+\frac{B}{m}\left(1-z\right)y', \quad y(0)=0, \quad y'(0)=B,
\label{mainEqAA}
\end{equation}
where the initial condition $y(0)=0$ was determined from the requirement that $U(\eta^*)=0$. It is also interesting that the condition $y'(0)=B$ is \emph{necessary}: it emerges from the structure of the equation (\ref{mainEqAA}) by setting $z=0$. It is worthy to stress the fact that the problem (\ref{mainEqAA}) is the \emph{same} for both sets of conditions (\ref{conds1})-(\ref{conds2}). The difference between those two cases will emerge when we calculate the wetting front position $eta^*$. 

Assuming the Taylor expansion
\begin{equation}
	y(z)=\sum_{k=1}^\infty a_k z^k, \quad a_k = \frac{y^{(k)}(0)}{k!},
\label{Taylor}
\end{equation}
equation (\ref{mainEqAA}) can be solved by a comparison of terms. Equivalently, but more computationally convenient, is to take the $n$th derivative of the both sides of (\ref{mainEqAA}) and use the Leibniz rule for derivative of the product. We then obtain
\begin{equation}
	\sum_{k=0}^n \binom{n}{k} \left(\frac{1}{m} y^{(k+1)}(z)y^{(n-k+1)}(z)+ y^{(k)}(z)y^{(n-k+2)}(z)\right)=A y^{(n)}(z)+\frac{B}{m}\sum_{k=0}^n \binom{n}{k} y^{(n-k+1)}(z)\frac{d^k}{dz^k}\left(1-z\right).
\end{equation}
Now, setting $z=0$ we have the recurrence relation for $a_n$
\begin{equation}
\begin{split}
	\sum_{k=0}^n \binom{n}{k} & \left(\frac{1}{m} (k+1)!(n-k+1)! \; a_{k+1}a_{n-k+1} + k! (n-k+2)! \; a_k a_{n-k+2}\right) \\ 
	&=A n! \; a_n+\frac{B}{m}\left((n+1)! \; a_{n+1}-n n! \; a_n\right),
\end{split}
\end{equation}
or, after simplification
\begin{equation}
\begin{split}
	&a_0=0, \quad a_1=B, \quad a_{n+1}=\frac{1}{B n (n+1)} \left[\left(A-\frac{Bn}{m}-\frac{2n}{m}a_2\right)a_n \right. \\
	&\left. -\sum_{k=2}^n\frac{1}{m} (k+1)(n-k+1)a_{k+1}a_{n-k+1}+(n-k+1)(n-k+2)a_k a_{n-k+2}\right], \quad n\geq 1
\end{split}
\label{recurr}
\end{equation}
which can be easily implemented in some scientific environment (here we used the convention that $\sum_{k=2}^1 = 0$). It is straightforward to calculate as much coefficients $a_n$ as necessary, for example
\begin{equation}
	a_2=\frac{Am-B}{2(1+m)}, \quad a_3=\frac{(A + B) m (B - A m)}{6 B(1 + m)^2 (1 + 2  m)}, \quad a_4=\frac{(A + B) m (B - A m) (B (3 + m) - A m (5 + 3 m))}{24 B^2 (1 + 
   m)^3 (1 + m (5 + 6 m))}.
\end{equation}

The Taylor series expansion (\ref{Taylor}) is often useful for $m$ of order of unity. For larger values we can apply the perturbation theory. Assuming that the solution $y$ of the equation (\ref{mainEqAA}) can be expressed in the form
\begin{equation}
	y(z)=y_0(z)+\frac{1}{m}y_1(z)+\cdots,
\label{pert}
\end{equation}
we obtain the following equations for $y_0$ and $y_1$
\begin{equation}
\begin{split}
	y_0 y_0''=A y_0, & \quad y_0(0)=0, \quad y'_0(0)=B, \\
	y_0'^2+y_0y_1''+y_1 y_0''=Ay_1+B(1-z)y_0', & \quad y_1(0)=0, \quad y_1'(0)=0.
\end{split}
\end{equation}
Immediately we obtain $y_0(z)=A z^2/2+Bz$ and therefore
\begin{equation}
	y_1(z)=-2(A+B)\left[\frac{z^2}{2}-\frac{B}{A}\left(\left(z+\frac{2B}{A}\right)\left(\ln\left(1+\frac{A}{2 B}z\right)-1\right)+\frac{2B}{A}\right)\right].
\end{equation}

The substitution (\ref{sub}) gives also a very convenient way of determining the wetting front position $\eta^*$. Here we must distinguish between the cases (\ref{conds1}) and (\ref{conds2}). For the former we have $1=U(0)=(m (\eta^*)^2 y(1))^{1/m}$ and thus
\begin{equation}
	\eta^*=\frac{1}{\sqrt{m y(1)}},
\label{eta1}
\end{equation}
For the case of (\ref{conds2}) we $-1=U^m(0)U'(0)=-m^{-1/m}(\eta^*)^{1-2/m}y^{1/m}(1)y'(1)$, hence
\begin{equation}
	\eta^* = \left(m^\frac{1}{m} y^\frac{1}{m}(1)y'(1)\right)^{-\frac{m}{m+2}}.
\label{eta2}
\end{equation}
Since in order to calculate $\eta^*$ we only need to know the values of $y$ and $y'$ at $z=1$, the series representation (\ref{Taylor}) can give very accurate approximations. When $m$ is large the perturbation approximation (\ref{pert}) yields better results. We note also that the formulas for the wetting front position $\eta^*$ (\ref{eta1})-(\ref{eta2}) are the only differences between two kinds of boundary conditions. The self-similar solution $U$ behaves in the same way for both of these setting. The difference appears only in the scaling and choice of $a$ and $b$ in (\ref{scale}). 

We can now go back to the substitution (\ref{subs2}) and write 
\begin{equation}
\begin{split}
	U(\eta)&=\left(m (\eta^*)^2 \sum_{k=0}^\infty a_k \left(1-\frac{\eta}{\eta^*}\right)^k\right)^\frac{1}{m}=\left(m (\eta^*)^2 \left(1-\frac{\eta}{\eta^*}\right) \sum_{k=0}^\infty a_{k+1} \sum_{j=0}^k \binom{k}{j} (-1)^j \left(\frac{\eta}{\eta^*}\right)^j\right)^\frac{1}{m} \\
	&=\left(m (\eta^*)^2 \left(1-\frac{\eta}{\eta^*}\right) \sum_{j=0}^\infty \left(\sum_{k=0}^j \binom{k}{j} a_{k+1} \right) \left(-\frac{\eta}{\eta^*}\right)^j\right)^\frac{1}{m},
\end{split}
\end{equation}
where $a_{k}$ can be calculated from the recurrence (\ref{recurr}) and $\eta^*$ is obtained for each of the boundary conditions (\ref{conds1})-(\ref{conds2}) from (\ref{eta1})-(\ref{eta2}) respectively. Although complicated, these formulas are readily computable in any scientific environment. It is also worth to mention the form of the derivative at $\eta=0$
\begin{equation}
	U'(0)=-(\eta^*)^{\frac{2}{m}-1}\left(m\sum_{k=1}^\infty a_k\right)^{\frac{1}{m}-1}\sum_{k=1}^\infty k a_k.
\label{U'0}
\end{equation}
Combining the above formula along with the value of $U$ taken from (\ref{conds1}) or (\ref{conds2}) we obtain an initial-value problem for $U$. That is, along with (\ref{U'0}) we have
\begin{equation}
	U(0)=1 \; (\text{for (\ref{conds1}})) \quad \text{or} \quad U(0)=\left((\eta^*)^{\frac{2}{m}-1}\left(m\sum_{k=1}^\infty a_k\right)^{\frac{1}{m}-1}\sum_{k=1}^\infty k a_k\right)^{-\frac{1}{m}} \; (\text{for (\ref{conds2}})) 
\end{equation}
The numerical solution of this problem is much simpler than the one with the unknown free boundary $\eta^*$. By truncating the series (\ref{Taylor}) we obtain the following approximations
\begin{equation}
\begin{split}
	U_1(\eta)&=\left(m (\eta_1^*)^2 a_1 \left(1-\frac{\eta}{\eta_1^*}\right) \right)^\frac{1}{m}, \quad U_2(\eta)=\left(m (\eta_2^*)^2 \left(1-\frac{\eta}{\eta_2^*}\right) \left(a_1+a_2-a_2\frac{\eta}{\eta_2^*}\right)\right)^\frac{1}{m}, \\
	U_3(\eta)&=\left(m (\eta_3^*)^2 \left(1-\frac{\eta}{\eta_3^*}\right) \left(a_1+a_2+a_3-(a_2+2a_3)\frac{\eta}{\eta_3^*}+a_3 \left(\frac{\eta}{\eta_3^*}\right)^2\right)\right)^\frac{1}{m}
\end{split}
\label{Ui}
\end{equation}
where by the subscript we denoted how many terms to take in (\ref{Taylor}) and (\ref{eta1})-(\ref{eta2}). These approximations can be used to determine the closed form of the \emph{cumulative moisture}
\begin{equation}
	I(t):=\int_0^\infty u(x,t) dx = \int_0^\infty t^a U\left(x t^{-b}\right)dx = t^{a+b} \int_0^{\eta^*} U(\eta)d\eta,
	\label{Infil}
\end{equation}
where $a$ and $b$ are chosen with respect to the considered boundary conditions (\ref{conds1})-(\ref{conds2}). By the use of (\ref{Ui}) we have
\begin{equation}
	I_1(t)= t^{a+b} \frac{m \eta_1^*}{1+m}\left(m (\eta_1^*)^2 a_1\right)^\frac{1}{m}, \; I_2(t)=t^{a+b}  \frac{m \eta_2^*}{1+m}\left(m (\eta_2^*)^2 a_1\right)^\frac{1}{m} \; _2 F_1\left(1+\frac{1}{m},-\frac{1}{m}; 2+\frac{1}{m}; - \frac{a_2}{a_1}\right),
	\label{Infil2}
\end{equation}
and in general
\begin{equation}
	I(t)=t^{a+b} \eta^* \left(m (\eta^*)^2 \right)^\frac{1}{m} \int_0^1 \left(\sum_{k=0}^\infty a_k y^k\right)^\frac{1}{m} dy.
\end{equation}
Here $_2 F_1$ denotes the hypergeometric function (see \cite{AS}). 
 
To illustrate the accuracy of our approximations we implement the finite difference numerical scheme proposed in \cite{Plo14} to solve the time-fractional porous medium equation (\ref{ERL2}) with conditions (\ref{cond4}) (the case of (\ref{cond3}) was investigated in the cited paper). This finite difference scheme is constructed as a weighted explicit-implicit method, where nonlinearity is linearized by the method borrowed from \cite{Ozi}. The simulated values of the diffusion front for (\ref{cond4}) are depicted on Fig. \ref{diff}. 

\begin{figure}[htb!]
	\centering
	\includegraphics[trim=300 280 300 280,scale=1]{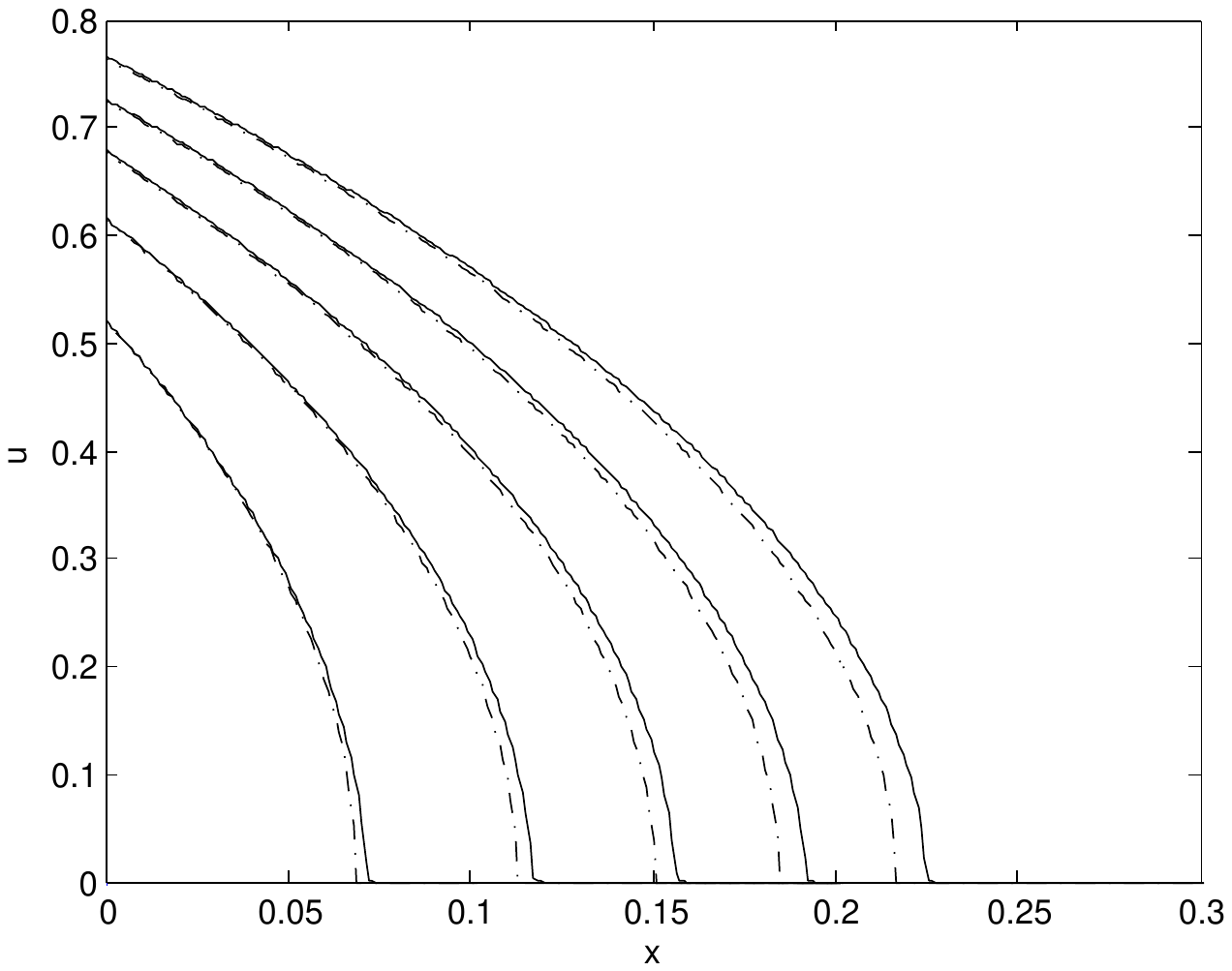}
	\caption{The solution of the anomalous diffusion equation (\ref{ERL2}) with conditions (\ref{cond4}) for $\alpha=0.95$, $m=2$ and times $t=0.02$, $0.04$, $0.06$, $0.08$, $0.1$ (curves from bottom to top). The dashed line represents the approximate solution $U_3$ from (\ref{Ui}).}
	\label{diff}
\end{figure}

As a additional check of the accuracy of the approximation we compute the wetting front position $\eta^*=\eta^*(t)$ and total infiltration (\ref{Infil}). The comparison between the approximate and exact (numerical) solution is presented of Fig. \ref{wetting}. We can see that the accuracy of approximation is decent, especially for small $\eta$ (what was suggested before).  

\begin{figure}[htb!]
	\centering
	\includegraphics[trim=130 280 110 280,scale=0.7]{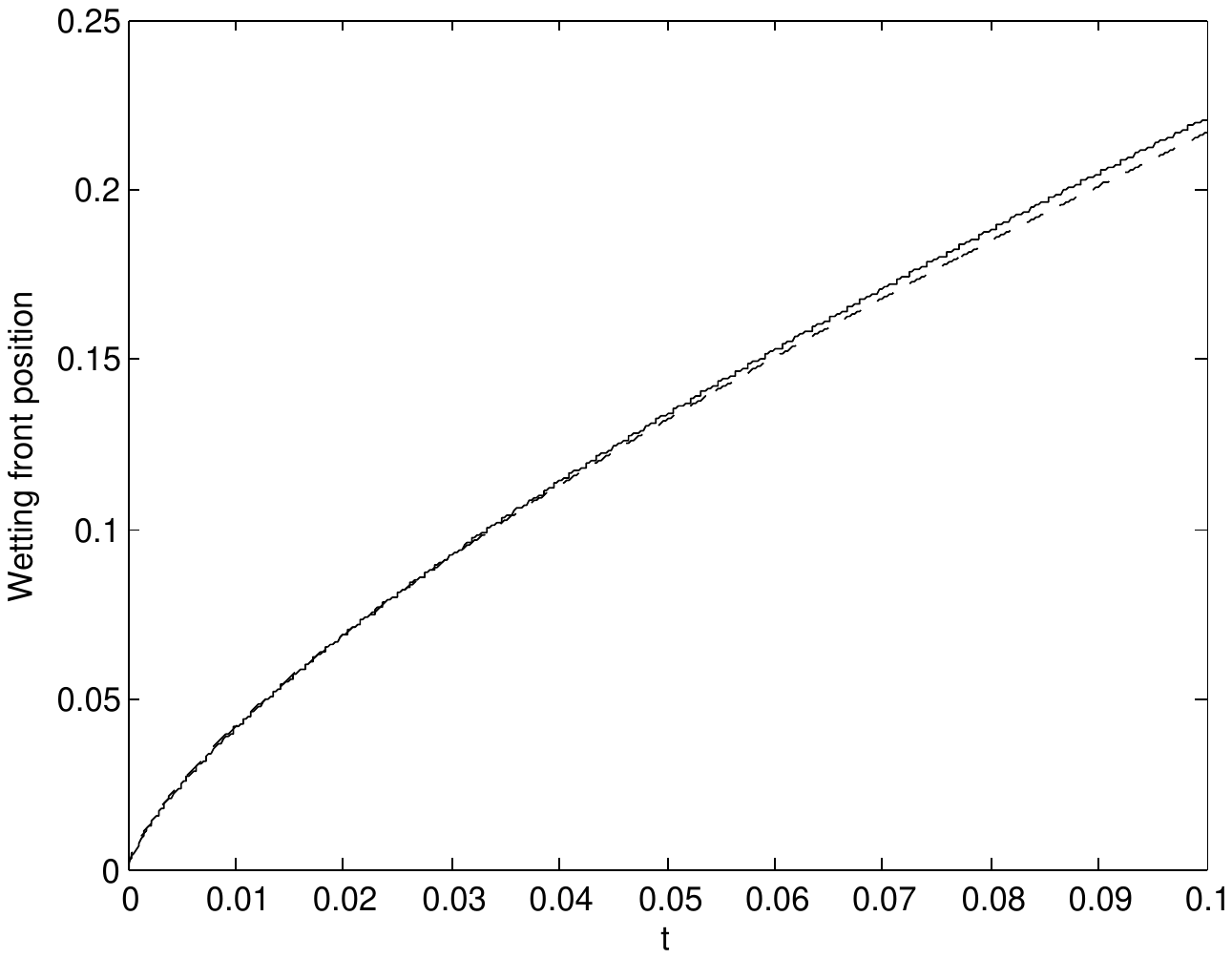}
	\includegraphics[trim=110 280 130 280,scale=0.7]{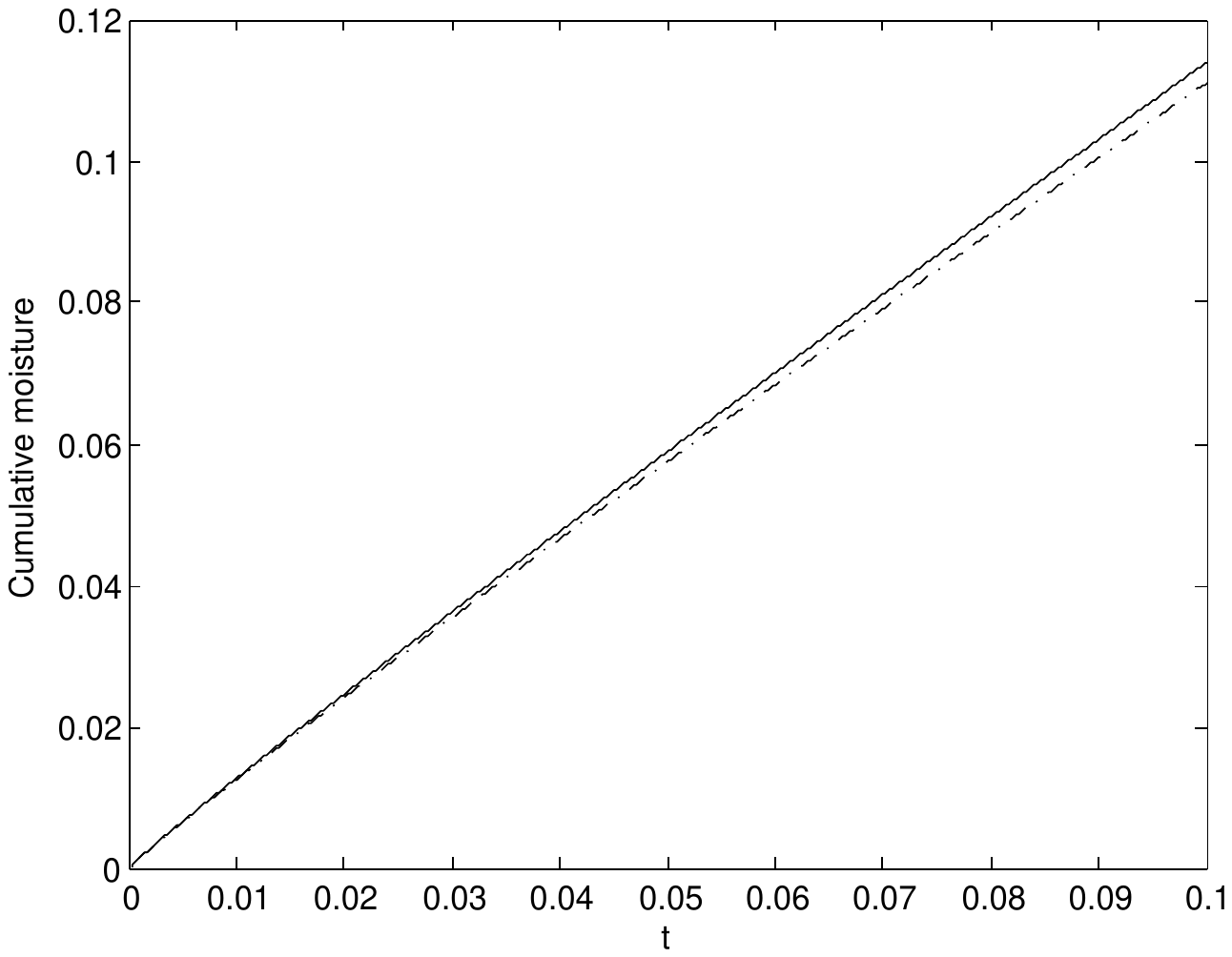}
	\caption{On the left: wetting front position $\eta^*=\eta^*(t)$; on the right: total cumulative infiltration $I=I(t)$ as in (\ref{Infil}). Solid lines represent the calculated numerical values while dashed lines are the approximations (\ref{eta2}) (for a three-term approximation) and (\ref{Infil2}) (for a two-term approximation) respectively. Here $\alpha=0.95$ and $m=2$.}
	\label{wetting}
\end{figure}

Finally, let us focus on the real data obtained from the experiment \cite{Aze06}. In that paper Authors reported a deviation from the standard Boltzmann scaling $x/\sqrt{t}$ when investigated a moisture ingress in porous zeolite. To model this phenomenon they used equation (\ref{ERL}) but instead of solving it, they approximated the diffusion coefficient $D=D(u)$. Moreover, they indicated that the moisture transport process in the porous medium is very dependent on the history. Hence, the idea of the model with time-fractional derivative of order $\alpha$ which is a nonlocal-in-time operator. The fitting results of the dimensional form (\ref{scaling1}) of the $U_3$ function (\ref{Ui}) with the boundary conditions (\ref{conds1}) are depicted on Fig. \ref{zeo}. The subdiffusive character of the medium is evident and the accuracy of our approximation is decent. 

\begin{figure}[htb!]
	\centering
	\includegraphics[scale=1]{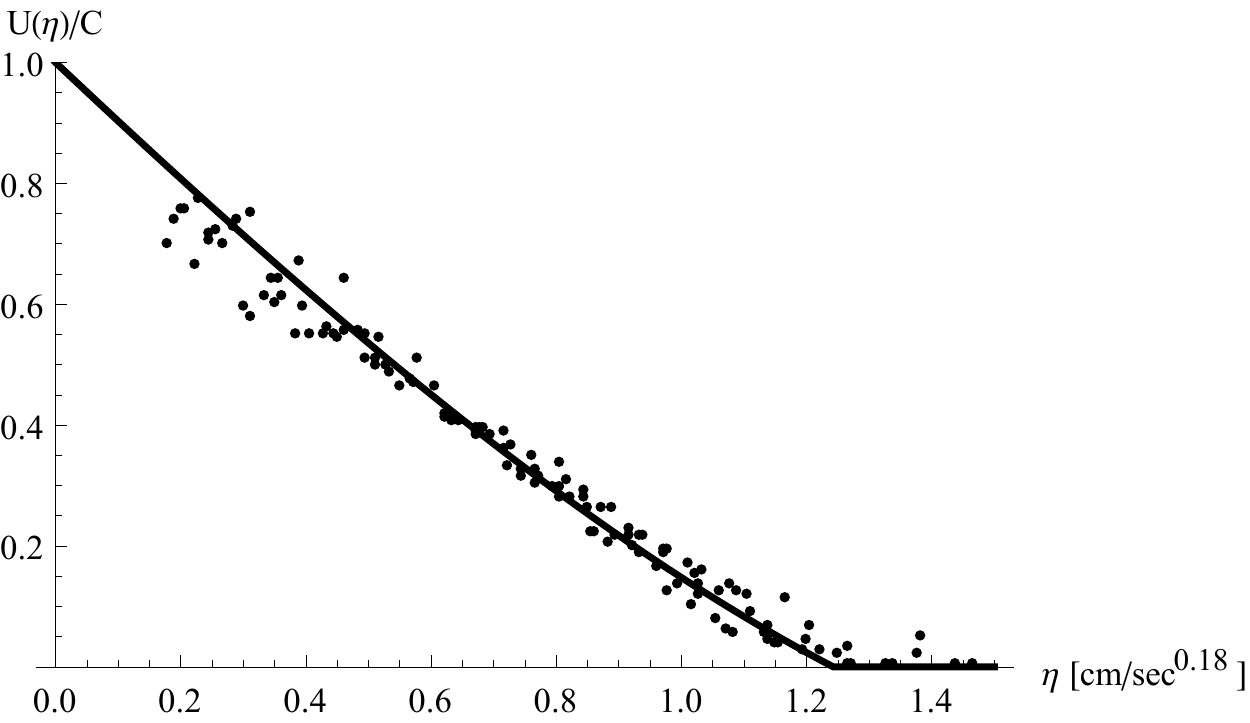}
	\caption{Fitting the function $U_3$ from (\ref{Ui}) to the self-similar moisture profile of the porous zeolite obtained from \cite{Aze06}. Here $\alpha=0.36$, $m=1$ and $D_0=0.4568$ [cm$^2$/sec$^{0.36}$]}
	\label{zeo}
\end{figure}

\section{Conclusion}
We have analyzed the so-called time-fractional porous medium equation, which is a mathematical description of the anomalous transport in some types of media (such as building materials). We have shown that the emergence of the temporal fractional derivative is a natural consequence of the trapping phenomena inside the medium, which lead to the subdiffusion. Our deterministic derivation is dual to the one done in the stochastic CTRW framework and thus shows a deep relation between these two approaches. In our setting, the nonlinear flux can be easily incorporated yielding the time-fractional Richards equation. Due to serious complexity of the resulting equation we have used the previously developed method of approximation and obtained decently accurate (approximate) solutions. All of the approximations have been verified numerically along with fitting to the real data representing wetting front position in the porous zeolite. Our results tentatively suggest that the time-fractional porous medium equation can be a sensible model for the anomalous diffusion in some porous media. But nevertheless, the topic has to be thoroughly investigated in order to state some concluding remarks. We hope that our research will help to provide a deeper understanding of the anomalous diffusion. 

\section*{Acknowledgment}
Author was co-financed by European Union within European Social Fund.

\end{document}